\documentclass[12pt, reqno]{amsart}
\usepackage{amsmath, amsthm, amscd, amsfonts, amssymb, graphicx, color}
\usepackage[bookmarksnumbered, colorlinks, plainpages]{hyperref}
\hypersetup{colorlinks=true,linkcolor=red, anchorcolor=green, citecolor=cyan, urlcolor=red, filecolor=magenta, pdftoolbar=true}

\newtheorem{theorem}{Theorem}[section]
\newtheorem{lemma}[theorem]{Lemma}
\newtheorem{proposition}[theorem]{Proposition}
\newtheorem{corollary}[theorem]{Corollary}
\theoremstyle{definition}
\newtheorem{definition}[theorem]{Definition}
\newtheorem{example}[theorem]{Example}

\theoremstyle{remark}
\newtheorem{remark}[theorem]{Remark}
\numberwithin{equation}{section}

\begin{document}

\setcounter{page}{1}

\title[The polar decomposition and centered operators]{The polar decomposition for adjointable operators on Hilbert $C^*$-modules and centered operators}

\author[N. Liu, W. Luo, Q. Xu]{Na Liu, Wei Luo, \MakeLowercase{and} Qingxiang Xu$^{*}$}

\address{$^{}$Department of Mathematics, Shanghai Normal University, Shanghai 200234, PR China.}
\email{\textcolor[rgb]{0.00,0.00,0.84}{liunana0616@163.com;
luoweipig1@163.com; qingxiang\_xu@126.com}}




\subjclass[2010]{Primary 46L08; Secondary 47A05.}

\keywords{Hilbert $C^*$-module, polar decomposition, centered operator.}


\begin{abstract}
Let $T$ be an adjointable operator between two Hilbert $C^*$-modules and $T^*$ be the adjoint operator of $T$. The polar decomposition of $T$ is characterized as $T=U(T^*T)^\frac12$ and  $\mathcal{R}(U^*)=\overline{\mathcal{R}(T^*)}$, where $U$ is a partial isometry,
$\mathcal{R}(U^*)$ and $\overline{\mathcal{R}(T^*)}$ denote the range of $U^*$ and the norm closure of the range of $T^*$,  respectively.
Based on this new characterization of the polar decomposition, an application to the study of centered operators is carried out.\\
\end{abstract} \maketitle

\section{Introduction}

Much progress has been made in the study of the polar decomposition for Hilbert space operators \cite{Furuta,Gesztesy-Malamud-Mitrea-Naboko,Halmos,Ichinose-Iwashita,Stochel-Szafraniec}.
Let $H, K$ be two Hilbert spaces and $\mathbb{B}(H, K)$ be the set of bounded linear operators from $H$ to $K$. For any $T\in\mathbb{B}(H, K)$, let $T^*$, $\mathcal{R}(T)$ and $\mathcal{N}(T)$ denote the conjugate operator, the range and the null space of $T$, respectively. It is well-known \cite{Halmos,Ichinose-Iwashita} that every operator $T\in\mathbb{B}(H,K)$ has the unique polar decomposition
\begin{equation}\label{equ:polar decomposition-Hilbert space case} T=U|T|\ \mbox{and}\ \mathcal{N}(T)=\mathcal{N}(U),
\end{equation}
where $|T|=(T^*T)^{\frac12}$ and $U\in\mathbb{B}(H,K)$ is a partial isometry. An alternative expression of \eqref{equ:polar decomposition-Hilbert space case} is
\begin{equation}\label{equ:polar decomposition-Hilbert space case-1}
T=U|T|\ \mbox{and}\ \overline{\mathcal{R}(T^*)}=\mathcal{R}(U^*),
\end{equation}
 since $\mathcal{N}(T)^{\bot}=\overline{\mathcal{R}(T^*)}$ and $\mathcal{N}(U)^{\bot}=\overline{\mathcal{R}(U^*)}=\mathcal{R}(U^*)$ in the Hilbert space case.
Note that if $H=K$, then $\mathbb{B}(H,H)$ abbreviated to $\mathbb{B}(H)$, is a von Neumann algebra. It follows from \cite[Proposition~2.2.9]{Pedersen} that the polar decomposition also works for elements in a von Neumann algebra. Nevertheless, it may be false for some elements in a general $C^*$-algebra; see \cite[Remark~1.4.6]{Pedersen}.

Both Hilbert spaces and $C^*$-algebras can be regarded as Hilbert $C^*$-modules, so one might study the polar decomposition in the general setting of Hilbert $C^*$-modules. An adjointable operator between Hilbert $C^*$-modules may have no polar decomposition unless some additional conditions are satisfied; see Lemma~\ref{lem:Wegge-Olsen} below for the details.
The polar decomposition for densely defined closed operators and unbounded operators are also considered in some literatures; see \cite{Frank-Sharifi,Gebhardt-Schm¨¹dgen,Gesztesy-Malamud-Mitrea-Naboko} for example.

The purpose of this paper is, in the general setting of adjointable operators on Hilbert $C^*$-modules, to provide a new insight into the polar decomposition theory and its applications. We will prove in Lemma~\ref{lem:one plus one implies one plus four} that five equalities appearing in Lemma~\ref{lem:Wegge-Olsen} (iii) can be in fact simplified to two equalities described in \eqref{equ:defn of polar decomposition for T-1}, which  are evidently the same as that in \eqref{equ:polar decomposition-Hilbert space case-1} when the underlying spaces are Hilbert spaces. It is remarkable that \eqref{equ:polar decomposition-Hilbert space case} is a widely used characterization of the polar decomposition for Hilbert space operators. Nevertheless, Example~\ref{ex:countable example} indicates that such a characterization of the polar decomposition is no longer true for adjointable operators on Hilbert $C^*$-modules. This leads us to figure out a modified version of \eqref{equ:polar decomposition-Hilbert space case}, which is stated in Theorem~\ref{thm:uniqueness of polar decomposition-1}. Note that the verification of the equivalence of Lemma~\ref{lem:Wegge-Olsen} (i) and (ii) is trivial, so it is meaningful to give another interpretation of Lemma~\ref{lem:Wegge-Olsen} (ii). We have managed to do that in Theorem~\ref{thm:existence condition for polar decomposition} (iii).

 One application of the polar decomposition is the study of centered operators on Hilbert spaces, which was initiated in \cite{Morrel} and generalized in \cite{Ito,Ito-Yamazaki-Yanagida}. Based on the new characterization \eqref{equ:defn of polar decomposition for T-1} of the polar decomposition for adjointable operators, some generalizations on centered operators are made in the framework of Hilbert $C^*$-modules.

The paper is organized as follows. Some elementary results on adjointable operators are provided in Section~\ref{sec:preliminaries}.
In Section~\ref{sec:the polar decomposition}, we focus on the study of the polar decomposition for adjointable operators on Hilbert $C^*$-modules.
As an application of the polar decomposition, centered operators are studied in Section~\ref{sec:defn of centered operators}.

\section{Some elementary results on adjointable operators}\label{sec:preliminaries}
Hilbert $C^*$-modules are generalizations of Hilbert spaces by allowing inner products to take values in some $C^{*}$-algebras instead of the complex  field. Let $\mathfrak{A}$ be a $C^*$-algebra. An
inner-product $\mathfrak{A}$-module is a linear space $E$ which is a right
$\mathfrak{A}$-module, together with a map $(x,y)\to \big<x,y\big>:E\times E\to
\mathfrak{A}$ such that for any $x,y,z\in E, \alpha, \beta\in \mathbb{C}$ and
$a\in \mathfrak{A}$, the following conditions hold:
\begin{enumerate}
\item[(i)] $\langle x,\alpha y+\beta
z\rangle=\alpha\langle x,y\rangle+\beta\langle x,z\rangle;$

\item[(ii)] $\langle x, ya\rangle=\langle x,y\rangle a;$

\item[(iii)] $\langle y,x\rangle=\langle x,y\rangle^*;$

\item[(iv)] $\langle x,x\rangle\ge 0, \ \mbox{and}\ \langle x,x\rangle=0\Longleftrightarrow x=0.$
\end{enumerate}

An inner-product $\mathfrak{A}$-module $E$ which is complete with respect to
the induced norm ($\Vert x\Vert=\sqrt{\Vert \langle x,x\rangle\Vert}$ for
$x\in E$) is called a (right) Hilbert $\mathfrak{A}$-module.

Suppose that $H$ and $K$ are two Hilbert $\mathfrak{A}$-modules, let ${\mathcal L}(H,K)$
be the set of operators $T:H\to K$ for which there is an operator $T^*:K\to
H$ such that $$\langle Tx,y\rangle=\langle x,T^*y\rangle \ \mbox{for any
$x\in H$ and $y\in K$}.$$  We call ${\mathcal
L}(H,K)$ the set of adjointable operators from $H$ to $K$. For any
$T\in {\mathcal L}(H,K)$, the range and the null space of $T$ are denoted by
${\mathcal R}(T)$ and ${\mathcal N}(T)$, respectively. In case
$H=K$, ${\mathcal L}(H,H)$ which is abbreviated to ${\mathcal L}(H)$, is a
$C^*$-algebra. Let ${\mathcal L}(H)_{sa}$ and ${\mathcal L}(H)_+$ be the set of  self-adjoint elements and positive elements
in $\mathcal{L}(H)$, respectively.

\begin{definition}
A closed
submodule $M$ of a Hilbert $\mathfrak{A}$-module $E$ is said to be
orthogonally complemented  if $E=M\dotplus M^\bot$, where
$$M^\bot=\big\{x\in E: \langle x,y\rangle=0\ \mbox{for any}\ y\in
M\big\}.$$
In this case, the projection from $H$ onto $M$ is denoted by $P_M$.
\end{definition}

Throughout the rest of this paper, $\mathfrak{A}$ is a $C^*$-algebra, $E,H$ and $K$ are three  Hilbert $\mathfrak{A}$-modules. Note that  $\mathcal{L}(H)$ is a $C^*$-algebra, so we begin with an elementary result on $C^*$-algebras.

\begin{definition}\label{defn:commutator of a-b}
 Let  $\mathfrak{B}$  be a $C^*$-algebra. The set of positive elements of $\mathfrak{B}$ is denoted by $\mathfrak{B}_+$.  For any $a,b\in \mathfrak{B}$,  let $[a,b]=ab-ba$ be the commutator of $a$ and $b$.
\end{definition}

\begin{proposition}\label{prop:commutative property extended-1} Let $\mathfrak{B}$ be a  $C^*$-algebra  and let $a,b\in \mathfrak{B}$ be such that $a=a^*$ and $[a,b]=0$.
Then  $[f(a),b]=0$  whenever $f$ is a continuous complex-valued function on the interval $[-\Vert a\Vert, \Vert a\Vert]$.
\end{proposition}
\begin{proof}We might as well assume that $\mathfrak{B}$ has a unit. Choose any sequence $\{p_n\}_{n=1}^\infty$ of polynomials  such that $p_n(t)\to f(t)$ uniformly on  the interval $[-\Vert a\Vert, \Vert a\Vert]$. Then $\Vert p_n(a)-f(a)\Vert\to 0$ as $n\to \infty$, hence
\begin{equation*}f(a)b=\lim_{n\to\infty}p_n(a)b=\lim_{n\to\infty}b\,p_n(a)=bf(a).\qedhere
\end{equation*}
\end{proof}

Next, we state some elementary results on the commutativity of adjointable operators. For any $\alpha>0$, the function $f(t)=t^\alpha$  is continuous on $[0,+\infty)$, so a direct application of Proposition~\ref{prop:commutative property extended-1} yields the following proposition:

\begin{proposition}\label{prop:commutative property extended-3} Let $S\in \mathcal{L}(H)$ and $T\in\mathcal{L}(H)_+$ be such that $[S,T]=0$. Then
$[S,T^\alpha]=0$ for any $\alpha>0$.
\end{proposition}

The technical result of this section is as follows:
\begin{proposition}\label{prop:SOT topology of the projection of the range-1}Let $T\in\mathcal{L}(H)_+$ be such that $\overline{\mathcal{R}(T)}$ is orthogonally complemented.
Then \begin{equation} \label{equ:characterization of projection is SOT-topology}\lim_{n\to\infty}\Vert T_nx-P_{\overline{\mathcal{R}(T)}}x\Vert=0\ \mbox{for all $x$ in $H$},\end{equation}
where $T_n=\left(\frac{1}{n}I+T\right)^{-1}T$  for each  $n\in \mathbb{N}$.
\end{proposition}

\begin{proof}For  each $n\in\mathbb{N}$,  the continuous function $f_n$ associated to the operator $T_n$ is given by
$$f_n(t)=\frac{t}{\frac{1}{n}+t}\ \mbox{for}\  t\in \mbox{sp}(T)\subseteq [0, \Vert T\Vert],$$ where $\mbox{sp}(T)$ is the spectrum of $T$. Then for each $n\in\mathbb{N}$,
\begin{eqnarray*}&&\Vert T_n\Vert=\max\left\{\big|f_n(t)\big|:t\in \mbox{sp}(T)\right\}\le 1;\\
&&\Vert T_nT-T\Vert=\max\left\{\big|tf_n(t)-t\big|:t\in \mbox{sp}(T)\right\}\le \frac{1}{n}.\end{eqnarray*}

Now, given any $x\in H$ and any $\varepsilon>0$, let $x=u+v$, where $u\in\overline{\mathcal{R}(T)}$ and $v\in\mathcal{N}(T)\subseteq \mathcal{N}(T_n)$ for any $n\in\mathbb{N}$. Choose $h\in H$ and $n_0\in\mathbb{N}$  such that
\begin{equation*}\Vert u-Th\Vert<\frac{\varepsilon}{3}\ \mbox{and}\ n_0>\frac{3(\Vert h\Vert+1)}{\varepsilon}.\end{equation*}
Then for any $n\in\mathbb{N}$ with $n\ge n_0$, we have
\begin{align*}&\Vert T_nx-P_{\overline{\mathcal{R}(T)}}x\Vert=\Vert T_nu-P_{\overline{\mathcal{R}(T)}}u\Vert\\
&\le\Vert T_nu-T_nTh\Vert+\Vert T_nTh-P_{\overline{\mathcal{R}(T)}}Th\Vert+\Vert P_{\overline{\mathcal{R}(T)}}Th-P_{\overline{\mathcal{R}(T)}}u\Vert\\
&\le\Vert T_n(u-Th)\Vert+\Vert (T_nT)h-Th\Vert+\Vert P_{\overline{\mathcal{R}(T)}}(Th-u)\Vert\\
&\le\Vert u-Th\Vert+\frac{1}{n} \Vert h\Vert+\Vert Th-u\Vert<\frac{\varepsilon}{3}+\frac{\varepsilon}{3}+\frac{\varepsilon}{3}=\varepsilon.
\end{align*}
This completes the proof of \eqref{equ:characterization of projection is SOT-topology}.
\end{proof}

Based on Proposition~\ref{prop:SOT topology of the projection of the range-1}, a result on the commutativity for adjointable operators can be provided as follows:

\begin{proposition}\label{prop:commutative property extended-33} Let $S\in \mathcal{L}(H)$ and let  $T\in\mathcal{L}(H)_+$ be such that $\overline{\mathcal{R}(T)}$ is orthogonally complemented. If  $[S,T]=0$, then $\left[S,  P_{\overline{\mathcal{R}(T)}}\right]=0$.
\end{proposition}

\begin{proof} Denote $P_{\overline{\mathcal{R}(T)}}$ simply by $P$. Since $[S,T]=0$, we have $[S,T_n]=0$, where $T_n$ ($n\in\mathbb{N}$) are given in Proposition~\ref{prop:SOT topology of the projection of the range-1}. It follows from \eqref{equ:characterization of projection is SOT-topology} that \begin{eqnarray*}P(Sx)=\lim_{n\to\infty}T_n(Sx)=\lim_{n\to\infty}S(T_nx)=S(Px)\ \mbox{for any $x\in H$}.\qedhere
\end{eqnarray*}
\end{proof}

We end this section by stating some range equalities for adjointable operators.

\begin{proposition}\label{prop:rang characterization-1} Let $A\in\mathcal{L}(H,K)$ and $B,C\in\mathcal{L}(E,H)$ be such that $\overline{\mathcal{R}(B)}=\overline{\mathcal{R}(C)}$. Then $\overline{\mathcal{R}(AB)}=\overline{\mathcal{R}(AC)}$.
\end{proposition}

\begin{proof} Let $x\in E$ be arbitrary. Since $Bx\in\overline{\mathcal{R}(C)}$, there exists a sequence $\{x_n\}$ in $E$ such that $Cx_n\to Bx$. Then $ACx_n\to ABx$, which means
$ABx\in\overline{\mathcal{R}(AC)}$, and thus $\mathcal{R}(AB)\subseteq \overline{\mathcal{R}(AC)}$ and furthermore
$\overline{\mathcal{R}(AB)}\subseteq \overline{\mathcal{R}(AC)}$. Similarly, we have $\overline{\mathcal{R}(AC)}\subseteq \overline{\mathcal{R}(AB)}$.
\end{proof}

\begin{lemma} \label{lem:Range Closure of Ta and T} {\rm\cite[Lemma 2.3]{Xu-Fang}} Let $T\in \mathcal{L}(H)_+$. Then $\overline{\mathcal{R}(T^{\alpha})}=\overline{\mathcal{R}(T)}$ for any $\alpha\in(0,1)$.
\end{lemma}

\begin{proposition} \label{prop:Range Closure of T alpha and T} Let $T\in \mathcal{L}(H)_+$. Then $\overline{\mathcal{R}(T^{\alpha})}=\overline{\mathcal{R}(T)}$ for any $\alpha>0$.
\end{proposition}

\begin{proof}In view of Lemma~\ref{lem:Range Closure of Ta and T}, we might as well assume that $\alpha>1$. Put $S=T^\alpha$. Then $S\in\mathcal{L}(H)_+$, so from Lemma~\ref{lem:Range Closure of Ta and T} we have
$$\overline{\mathcal{R}(T)}=\overline{\mathcal{R}(S^{\frac{1}{\alpha}})}=\overline{\mathcal{R}(S)}=\overline{\mathcal{R}(T^{\alpha})}.\qedhere$$
\end{proof}

\section{The polar decomposition for adjointable operators}\label{sec:the polar decomposition}

In this section, we study the polar decomposition for adjointable operators on Hilbert $C^*$-modules.

\begin{definition} Recall that an element $U$ of $\mathcal{L}(H,K)$ is said to be a partial isometry if $U^*U$ is a projection in
$\mathcal{L}(H)$.
\end{definition}

\begin{proposition}\label{prop:basic property of partial isometry}{\rm \cite[Lemma 2.1]{Xu-Fang}}\ Let $U\in\mathcal{L}(H,K)$  be a partial isometry. Then $U^*$ is also a partial isometry which satisfies $UU^*U=U$.
\end{proposition}

\begin{lemma}\label{lem:Range closure of TT and T} {\rm\cite[Proposition 3.7]{Lance}}
Let $T\in\mathcal{L}(H,K)$. Then $\overline {\mathcal{R}(T^*T)}=\overline{ \mathcal{R}(T^*)}$ and $\overline {\mathcal{R}(TT^*)}=\overline{ \mathcal{R}(T)}$.
\end{lemma}

\begin{definition} For any $T\in\mathcal{L}(H,K)$, let $|T|$ denote the square root of $T^*T$. That is, $|T|=(T^*T)^\frac12$ and $|T^*|=(TT^*)^\frac12$.
\end{definition}

\begin{lemma}\label{lem:Wegge-Olsen}{\rm\cite[Proposition~15.3.7]{Wegge-Olsen}}\ Let $T\in\mathcal{L}(E)$. Then the following statements are equivalent:
\begin{enumerate}
\item[{\rm (i)}] $E=\mathcal{N}(|T|)\oplus \overline{\mathcal{R}(|T|)}$ and $E=\mathcal{N}(T^*)\oplus \overline{\mathcal{R}(T)}$;

\item[{\rm (ii)}] Both  $\overline{\mathcal{R}(T)}$
and $\overline{\mathcal{R}(|T|)}$  are orthogonally complemented;

\item[{\rm (iii)}]
$T$ has the polar decomposition $T=U|T|$, where $U\in\mathcal{L}(E)$ is a partial isometry such that
\begin{eqnarray}\begin{split}\label{eqn:four equalities of rangges and null spaces-1}\mathcal{N}(U)&=\mathcal{N}(T), \mathcal{N}(U^*)=\mathcal{N}(T^*),\\
\mathcal{R}(U)&=\overline{\mathcal{R}(T)}, \mathcal{R}(U^*)=\overline{\mathcal{R}(T^*)}.\end{split}
\end{eqnarray}
\end{enumerate}
\end{lemma}

\begin{lemma}\label{lem:one plus one implies one plus four} Let $T\in\mathcal{L}(H,K)$ be such that $\overline{\mathcal{R}(T^*)}$ is orthogonally complemented, and let $U\in\mathcal{L}(H,K)$ be a partial isometry such that
\begin{equation}\label{equ:defn of polar decomposition for T-1}T=U|T| \ \mbox{and}\  U^*U=P_{\overline{\mathcal{R}(T^*)}}.
\end{equation}
Then $\overline{\mathcal{R}(T)}$ is also orthogonally complemented, and all equations in \eqref{eqn:four equalities of rangges and null spaces-1} are satisfied. Furthermore, the following equations are also valid:
\begin{eqnarray}\label{equ:the polar decomposition of T star-pre stage}&&T^*=U^*|T^*|\ \mbox{and}\ UU^*=P_{\overline{\mathcal{R}(T)}},\\
\label{eqn:key relationship between two 1/2}&&|T^*|=U|T| U^*\ \mbox{and}\ U|T|=|T^*|U.
\end{eqnarray}
\end{lemma}

\begin{proof} By  Proposition~\ref{prop:Range Closure of T alpha and T} and  Lemma~\ref{lem:Range closure of TT and T}, we have
\begin{equation}\label{equ:the closures of one half-1}\overline{\mathcal{R}(|T|)}=\overline{\mathcal{R}(T^*T)}=\overline{\mathcal{R}(T^*)}=\mathcal{R}(U^*U),\end{equation}
which gives  by  Proposition~\ref{prop:rang characterization-1} that
$$\overline{\mathcal{R}(T)}=\overline{\mathcal{R}(U|T|)}=\overline{\mathcal{R}(UU^*U)}=\mathcal{R}(UU^*),$$
hence $\overline{\mathcal{R}(T)}$ is orthogonally complemented such that the second equation in \eqref{equ:the polar decomposition of T star-pre stage}
is satisfied. Furthermore,
$$TT^*=U|T|\cdot |T|U^*=(U|T|U^*)^2,$$
hence the first equation in \eqref{eqn:key relationship between two 1/2} is satisfied. As a result,
\begin{eqnarray*}&&U^*|T^*|=(U^*U|T|)U^*=|T|U^*=(U|T|)^*=T^*,\\
&&U|T|=T=(T^*)^*=(U^*|T^*|)^*=|T^*|U.
\end{eqnarray*}
This completes the proof of \eqref{equ:the polar decomposition of T star-pre stage} and \eqref{eqn:key relationship between two 1/2}. Finally, equations stated in \eqref{eqn:four equalities of rangges and null spaces-1} can be derived directly from the second equations in \eqref{equ:defn of polar decomposition for T-1} and \eqref{equ:the polar decomposition of T star-pre stage}, respectively.
\end{proof}

\begin{lemma}\label{lem:xu and fang-LAA}{\rm\cite[Theorem 3.1]{Xu-Fang}}\ Let $T\in \mathcal{L}(H,K)$ be such that $\overline{\mathcal{R}(T^*)}$ is orthogonally complemented. If $\mathcal{R}(|T^*|)\subseteq\mathcal{R}(T)$, then the following statements are valid:
\begin{enumerate}
\item[{\rm (i)}] $\mathcal{R}(|T^*|)=\mathcal{R}(T)$;

\item[{\rm (ii)}] $\mathcal{R}(|T|)=\mathcal{R}(T^*)$;

\item[{\rm (iii)}] $\overline{\mathcal{R}(T)}\,  \mbox{is orthogonally complemented}$.

\end{enumerate}
\end{lemma}

\begin{theorem}\label{thm:existence condition for polar decomposition}  Let $T\in \mathcal{L}(H,K)$. Then the following statements are equivalent:
\begin{enumerate}
\item[{\rm (i)}] $\overline{\mathcal{R}(T)}$ and $\overline{\mathcal{R}(T^*)}$ are both orthogonally complemented;
\item[{\rm (ii)}]  $\overline{\mathcal{R}(T^*)}$ is orthogonally complemented and \eqref{equ:defn of polar decomposition for T-1} is satisfied for some partial isometry $U\in\mathcal{L}(H,K)$;
\item[{\rm (iii)}]  $\overline{\mathcal{R}(T^*)}$ is orthogonally complemented,  $\mathcal{R}(|T|)=\mathcal{R}(T^*)$ and $\mathcal{R}(|T^*|)=\mathcal{R}(T)$.
\end{enumerate}
\end{theorem}

\begin{proof} The implications of (ii)$\Longrightarrow$(i) and (iii)$\Longrightarrow$(i) follow from Lemmas~\ref{lem:one plus one implies one plus four} and \ref{lem:xu and fang-LAA}, respectively.

$``\mbox{(i)}\Longrightarrow\mbox{(ii)}"$: Let $E=H\oplus K$ and $\widetilde{T}=\left(
                                                     \begin{array}{cc}
                                                       0 & 0 \\
                                                       T & 0 \\
                                                     \end{array}
                                                   \right)\in\mathcal{L}(E)$. Then
both $\overline{\mathcal{R}(\widetilde{T})}$ and $\overline{\mathcal{R}(\widetilde{T}^*)}$ are orthogonally complemented, hence
by Lemma~\ref{lem:Wegge-Olsen} there exists a partial isometry $\widetilde{U}=\left(
                                                                                \begin{array}{cc}
                                                                                  U_{11} & U_{12} \\
                                                                                  U & U_{22} \\
                                                                                \end{array}
                                                                              \right)
\in\mathcal{L}(E)$ such that
\begin{equation*}\widetilde{T}=\widetilde{U}|\widetilde{T}|, \mathcal{R}(\widetilde{U})=\overline{\mathcal{R}(\widetilde{T})}=\{0\}\oplus \overline{\mathcal{R}(T)}\ \mbox{and}\ \mathcal{R}(\widetilde{U}^*)=\overline{\mathcal{R}(\widetilde{T}^*)}=\overline{\mathcal{R}(T^*)}\oplus \{0\},
\end{equation*}
which leads to  $\widetilde{U}=\left(
                                 \begin{array}{cc}
                                   0 & 0 \\
                                   U & 0 \\
                                 \end{array}
                               \right)$, hence $U$ is a partial isometry satisfying  \eqref{equ:defn of polar decomposition for T-1}.

$``\mbox{(ii)}\Longrightarrow\mbox{(iii)}"$: By \eqref{equ:defn of polar decomposition for T-1}--\eqref{eqn:key relationship between two 1/2}, we have
\begin{equation*}T^*=(U|T|)^*=|T|U^*\ \mbox{and}\ T^*U=U^*|T^*|\cdot U=U^*\cdot U|T|U^*\cdot U=|T|,
\end{equation*}
which obviously lead to $\mathcal{R}(|T|)=\mathcal{R}(T^*)$. Replacing $T,U$ with $T^*, U^*$,
we obtain $\mathcal{R}(|T^*|)=\mathcal{R}(T)$.
\end{proof}

\begin{lemma}\label{lem:uniqueness of polar decomposition-1} Let $T\in \mathcal{L}(H,K)$ be such that $\overline{\mathcal{R}(T^*)}$ is orthogonally complemented. If $U,V\in \mathcal{L}(H,K)$ are given such that $U|T|=V|T|$ and $U^*U=V^*V=P_{\overline{\mathcal{R}(T^*)}}$, then $U=V$.
\end{lemma}

\begin{proof} The equation $U|T|=V|T|$ together with \eqref{equ:the closures of one half-1} yields  $UP_{\overline{\mathcal{R}(T^*)}}=VP_{\overline{\mathcal{R}(T^*)}}$, hence
$$U=U(U^*U)=UP_{\overline{\mathcal{R}(T^*)}}=VP_{\overline{\mathcal{R}(T^*)}}=V(V^*V)=V.\qedhere$$
\end{proof}

\begin{definition}\label{defn:defn of polar decomposition} The polar decomposition of $T\in\mathcal{L}(H,K)$ can be characterized as
\begin{equation}\label{equ:two conditions of polar decomposition}T=U|T|\ \mbox{and}\ U^*U=P_{\overline{\mathcal{R}(T^*)}},\end{equation}
where $U\in\mathcal{L}(H,K)$ is a partial isometry.
\end{definition}

\begin{remark}\label{rem:existence of polar decomposition} It follows from Theorem~\ref{thm:existence condition for polar decomposition} and Lemma~\ref{lem:uniqueness of polar decomposition-1}  that $T\in\mathcal{L}(H,K)$ has the (unique) polar decomposition if and only if  $\overline{\mathcal{R}(T^*)}$ and $\overline{\mathcal{R}(T)}$ are both orthogonally complemented. In this case, $T^*=U^*|T^*|$ is the polar decomposition of $T^*$.
\end{remark}

A slight generalization of \eqref{eqn:key relationship between two 1/2} is as follows:

\begin{lemma}\label{lem:relationship between |T| and |T-star| alpha} Let $T=U|T|$ be the polar decomposition of  $T\in \mathcal{L}(H,K)$. Then for any $\alpha>0$, the following statements are valid:
\begin{enumerate}
\item[{\rm (i)}] $U|T|^\alpha U^*=(U|T|U^*)^\alpha=|T^*|^\alpha$;
\item[{\rm (ii)}] $U|T|^\alpha=|T^*|^\alpha U$;
\item[{\rm (iii)}] $U^*|T^*|^\alpha U=(U^*|T^*|U)^\alpha=|T|^\alpha$.
\end{enumerate}
\end{lemma}

\begin{proof}(i) Since $U^*U |T|=|T|$, we have
\begin{equation}\label{equ:n-th power of |T-sta|}\big(U|T|U^*\big)^n=U|T|^nU^*\ \mbox{for any $n\in\mathbb{N}$}.\end{equation}
Let $f(t)=t^\alpha$ and choose any sequence $\{P_m\}_{m=1}^\infty$ of polynomials  such that $P_m(0)=0$ ($\forall m\in\mathbb{N}$), and $P_m(t)\to f(t)$ uniformly
on the interval $\left[0,\big\Vert |T|\big\Vert\right]$. Then from \eqref{eqn:key relationship between two 1/2} and \eqref{equ:n-th power of |T-sta|}, we have
\begin{eqnarray*}U|T|^\alpha U^*&=&U\,f(|T|)U^*=\lim_{m\to\infty}UP_m(|T|)U^*=\lim_{m\to\infty}P_m\big(U|T|U^*\big)\\
&=&f\big(U|T|U^*\big)=(U|T|U^*)^\alpha=|T^*|^\alpha.
\end{eqnarray*}

(ii) By Proposition~\ref{prop:Range Closure of T alpha and T} and Lemma~\ref{lem:Range closure of TT and T}, we have
$$\overline{\mathcal{R}(|T|^\alpha)}=\overline{\mathcal{R}(T^*T)}=\overline{\mathcal{R}(T^*)},$$
and thus $U^*U |T|^\alpha=|T|^\alpha$. Taking $*$-operation, we get $|T|^\alpha=|T|^\alpha U^*U$.
It follows from (i) that
$$U|T|^\alpha=U\big(|T|^\alpha U^*U\big)=\big(U|T|^\alpha U^*\big)U=|T^*|^\alpha U.$$

(iii) Since $T^*=U^*|T^*|$ is the polar decomposition of $T^*$, the conclusion follows immediately from (i) by replacing the pair $(U,T)$ with $(U^*,T^*)$.
\end{proof}

Before ending this section, we provide a  criteria for the polar decomposition  as follows:
\begin{theorem}\label{thm:uniqueness of polar decomposition-1} Let $T\in\mathcal{L}(H,K)$ be such that $\overline{\mathcal{R}(T^*)}$ is orthogonally complemented. Let $U\in \mathcal{L}(H,K)$ be a partial isometry which satisfies
\begin{equation}\label{equ:two conditions for polar decomposition-1}T=U|T|\ \mbox{and}\ \mathcal{N}(T)\subseteq \mathcal{N}(U).
\end{equation}Then $T=U|T|$ is the polar decomposition of $T$.
\end{theorem}

\begin{proof} By assumption, $Q=U^*U$ is a projection. For any $x\in H$, we have
\begin{equation*}\langle |T|x,|T|x\rangle=\langle Tx,Tx\rangle=\langle U|T|x,U|T|x\rangle=\langle Q|T|x,|T|x\rangle,
\end{equation*}and thus
\begin{equation*}\big\Vert (I-Q)|T|x\big\Vert^2=\big\Vert\langle (I-Q)|T|x,|T|x\rangle\big\Vert=0,
\end{equation*}
hence  $Q|T|=|T|$. It follows that $\overline{\mathcal{R}(T^*)}=\overline{\mathcal{R}(|T|)}\subseteq \mathcal{R}(Q)$. On the other hand, by assumption we have
\begin{eqnarray*}&&\overline{\mathcal{R}(T^*)}=\mathcal{N}(T)^\bot\supseteq \mathcal{N}(U)^\bot=\mathcal{N}(Q)^\bot=\mathcal{R}(Q),
\end{eqnarray*}
hence $\overline{\mathcal{R}(T^*)}=\mathcal{R}(Q)$ and thus $Q=P_{\overline{\mathcal{R}(T^*)}}$.
\end{proof}

\begin{remark}\label{rem:differences 1} Let $T\in\mathcal{L}(H,K)$, where  $H$ and $K$ are both Hilbert spaces.
 In this case $\overline{\mathcal{R}(T^*)}$ is always orthogonally complemented, so if $U$ is a partial isometry such that \eqref{equ:two conditions for polar decomposition-1} is satisfied, then $U|T|$ is exactly the polar decomposition of $T$.

 Unlike the assertion given in \cite[P.\,3400]{Karizaki-Hassani-Amyari}, the same is not true for general Hilbert $C^*$-modules $H$ and $K$,
 since $\overline{\mathcal{R}(T^*)}$ can be not orthogonally complemented for some $T\in\mathcal{L}(H,K)$.
 Indeed, there exist a Hilbert $C^*$-module $H$, and an adjointable $T$ and a partial isometry $U$ on $H$ such that \eqref{equ:polar decomposition-Hilbert space case} is satisfied, whereas $T$ has no polar decomposition. Such an example is as follows:
\end{remark}

\begin{example}\label{ex:countable example}{\rm
Let $H$ be any countably infinite-dimensional Hilbert space, $\mathcal{L}(H)$ and $\mathcal{C}(H)$ be the set of bounded
 linear operators and compact operators on $H$, respectively.
 Given any  orthogonal normalized basis $\{e_n: n\in\mathbb{N}\}$ for $H$, let $S\in\mathcal{C}(H)$ be defined by
$$S(e_n)=\frac{1}{n}e_n,\ \mbox{for any $n\in\mathbb{N}$}.$$ Clearly, $S$ is a positive element in  $\mathcal{C}(H)$.
Let $K=\mathfrak{A}=\mathcal{L}(H)$. With the inner product given by $$\big<X,Y\big>=X^*Y \ \mbox{for any $X,Y\in K$},$$
$K$ is a  Hilbert $\mathfrak{A}$-module.

Let $T: K\to K$ be defined by $T(X)=SX$ for any $X\in K$. Clearly, $T\in \mathcal{L}(K)_+$ and $\mathcal{R}(T)\subseteq \mathcal{C}(H)$.
Given any $n\in\mathbb{N}$, let $P_n$ be the projection from $H$ onto the linear subspace spanned by $\{e_1,e_2,\cdots,e_n\}$.
Let $X_n\in K$ be defined by
$$X_n(e_j)=\left\{
         \begin{array}{ll}
           je_j, &\mbox{if $1\le j\le n$}, \\
           0, &\mbox{otherwise.} \\
         \end{array}
       \right..$$

It is obvious that $T(X_n)=P_n$, which implies that $\overline{\mathcal{R}(T)}=\mathcal{C}(H)$, hence $\overline{\mathcal{R}(T)}^\bot=\{0\}$,
therefore $\overline{\mathcal{R}(T)}$  fails to be orthogonally complemented. By Remark~\ref{rem:existence of polar decomposition}, we conclude that $T$ has no polar decomposition. Furthermore, given any $X\in K$ such that $T(X)=SX=0$, then $X=0$ since $S$ is injective. It follows that $\mathcal{N}(T)=\{0\}$.

Now, let $U$ be the identity operator on $K$. Then since $T$ is positive, we have $T=U|T|$ and
$\mathcal{N}(U)=\mathcal{N}(T)$, whereas $T$ has no polar decomposition.
}\end{example}

\section{Characterizations of centered operators}\label{sec:defn of centered operators}
In this section, we study centered operators in the  general setting of Hilbert $C^*$-modules.

\begin{definition}\cite{Morrel}\ An element $T\in\mathcal{L}(H)$ is said to be centered if the following sequence
$$\cdots, T^3(T^3)^*, T^2(T^2)^*,TT^*,T^*T,(T^2)^*T^2,(T^3)^*T^3,\cdots$$
consists of mutually commuting operators.
\end{definition}

We began with a cancelation law introduced in \cite[Lemma~3.7]{Ito}. Let $T=U|T|$ be the polar decomposition of $T\in\mathcal{L}(H)$.
Suppose that $n\in\mathbb{N}$ is given such that
\begin{equation}\label{equ:commutative condition for first n terms}\big[U^k|T|(U^k)^*,|T|\big]=0 \quad \mbox{for $1\le k\le n$}.\end{equation} Then
by Proposition~\ref{prop:commutative property extended-33} we have \begin{equation}\label{equ:commutative condition for first n terms+++}\big[U^k|T|(U^k)^*,U^*U\big]=0 \quad \mbox{for $1\le k\le n$},\end{equation}
hence for any $s,t\in\mathbb{N}$ with $1\le t\le s\le n+1$, we have
\begin{equation}\label{equ:assumption commutative condition+0} U^s|T|(U^s)^*U^t=U^s|T|(U^{s-t})^*\ \mbox{and}\ (U^t)^*U^s|T|(U^s)^*=U^{s-t}|T|(U^s)^*.
\end{equation}
Indeed, by \eqref{equ:commutative condition for first n terms} and \eqref{equ:commutative condition for first n terms+++} we have
\begin{eqnarray*}U^s|T|(U^s)^*U^t&=&U\cdot U^{s-1}|T|(U^{s-1})^*\cdot U^*U\cdot U^{t-1}\\
&=&U\cdot U^*U\cdot  U^{s-1}|T|(U^{s-1})^*\cdot U^{t-1}\\
&=&U\cdot U^{s-1}|T|(U^{s-1})^*\cdot U^{t-1}\\
&=&U^2\cdot U^{s-2}|T|(U^{s-2})^*\cdot U^{t-2}\\
&=&\cdots=U^s|T|(U^{s-t})^*.
\end{eqnarray*}
Taking $*$-operation, we get the second equation in \eqref{equ:assumption commutative condition+0}.

Now we are ready to state the technical lemma of this section,  which is a modification of \cite[Lemma~4.2]{Ito}.
\begin{lemma}\label{lem:Ito's 1st technical lemma}\ Suppose that $T=U|T|$ is the polar decomposition of $T\in\mathcal{L}(H)$. Let $n\in\mathbb{N}$ be given such that
\begin{equation}\label{equ:key commutative  assumption-00}\big[U^k|T|(U^k)^*, |T^l|\big]=0\quad  \mbox{for any $k,l\in\mathbb{N}$ with $k+l\le n+1$}.
\end{equation}
Then the following statements are equivalent:
\begin{enumerate}
\item[{\rm (i)}] $\big[U^s|T|(U^s)^*, |T^t|\big]=0$ for some $s,t\in\mathbb{N}$ with $s+t=n+2$;
\item[{\rm (ii)}] $\big[U^s|T|(U^s)^*, |T^t|\big]=0$ for any $s,t\in\mathbb{N}$ with $s+t=n+2$.
\end{enumerate}
\end{lemma}
\begin{proof} Let $s,t\in \mathbb{N}$ be such that $s+t=n+2$. Put
\begin{equation}\label{equ:defn of At and Bt}A_t=|T^t|^2\cdot U^s|T|(U^s)^*\ \mbox{and}\ B_t=U^s|T|(U^s)^*\cdot |T^t|^2.
\end{equation}
Then clearly,
\begin{equation}\label{equ:the meaning of At equals Bt}\big[U^s|T|(U^s)^*, |T^t|\big]=0\Longleftrightarrow A_t=B_t.\end{equation}
Substituting $k=1$ and $U|T|U^*=|T^*|$ into \eqref{equ:key commutative  assumption-00} yields
\begin{equation*}\big[|T^*|, |T^l|\big]=0 \quad \mbox{for $1\le l\le n$}, \end{equation*}
which leads by
Proposition~\ref{prop:commutative property extended-33} to
\begin{equation}\label{equ:commutative respect to UU-star-1}\big[UU^*, |T^l|\big]=0 \quad \mbox{for $1\le l\le n$}.\end{equation}

Note that $t=n+2-s\le n+1$, so if $t\ge 2$, then by \eqref{equ:key commutative  assumption-00},
\begin{align}|T^t|^2&=T^*\cdot (T^{t-1})^*T^{t-1}\cdot T=|T|U^*\cdot |T^{t-1}|^2\cdot U|T|\nonumber\\
&=U^*\cdot U|T|U^*\cdot |T^{t-1}|^2\cdot U|T|
=U^*\cdot |T^{t-1}|^2 \cdot    U|T|U^* \cdot U|T|\nonumber\\
\label{equ:computation of square root of T t}&=U^*\cdot |T^{t-1}|^2 \cdot    U|T|^2.
\end{align}
Assume now that $t\ge 2$. Then $s+1=(n+2-t)+1\le n+1$, hence by \eqref{equ:defn of At and Bt}, \eqref{equ:computation of square root of T t},  \eqref{equ:key commutative  assumption-00}
 with $l=1$ and  \eqref{equ:assumption commutative condition+0},
   we have
\begin{align}A_t&=U^*|T^{t-1}|^2 U\cdot |T|^2\cdot U^s|T|(U^s)^*\nonumber\\
\label{eqn:relationship between At and A t-1-0}&=U^*|T^{t-1}|^2 U\cdot U^s|T|(U^s)^* \cdot |T|^2\\
\label{eqn:relationship between At and A t-1}&=U^*\cdot A_{t-1}\cdot U|T|^2.
\end{align}
Similarly,
\begin{align} B_t&=U^s|T|(U^s)^*\cdot U^*UU^* |T^{t-1}|^2 U|T|^2\nonumber\\
&=U^*U\cdot U^s|T|(U^s)^*U^* |T^{t-1}|^2 U|T|^2\nonumber\\
&=U^*\cdot U^{s+1}|T|(U^{s+1})^*|T^{t-1}|^2 \cdot U|T|^2\nonumber\\
\label{eqn:relationship between Bt and B t-1}&=U^*\cdot B_{t-1}\cdot U|T|^2.
\end{align}

It follows from \eqref{eqn:relationship between At and A t-1} and  \eqref{eqn:relationship between Bt and B t-1} that $A_t=B_t$ whenever $A_{t-1}=B_{t-1}$.
Suppose on the contrary that $A_t=B_t$. Then by \eqref{eqn:relationship between At and A t-1-0}, \eqref{equ:commutative respect to UU-star-1}
and \eqref{equ:defn of At and Bt}, we have
\begin{align*}UA_t&=UU^*\cdot |T^{t-1}|^2 U\cdot U^s|T|(U^s)^* \cdot |T|^2=|T^{t-1}|^2\cdot UU^*U\cdot U^s|T|(U^s)^* |T|^2\\
&=|T^{t-1}|^2\cdot U\cdot U^s|T|(U^s)^* \cdot |T|^2=|T^{t-1}|^2\cdot U^{s+1}|T|(U^{s+1})^* \cdot U|T|^2\\
&=A_{t-1}\cdot U|T|^2
\end{align*}
Furthermore, it can be deduced directly from \eqref{eqn:relationship between Bt and B t-1} and \eqref{equ:defn of At and Bt} that
\begin{align*}UB_t=B_{t-1}\cdot U|T|^2.
\end{align*}
As a result, we obtain
\begin{equation*}A_{t-1}\cdot U|T|^2=B_{t-1}\cdot U|T|^2,
\end{equation*}
which gives
\begin{equation*}A_{t-1}=A_{t-1}UU^*=B_{t-1}UU^*=B_{t-1},
\end{equation*}
since $\overline{\mathcal{R}(U|T|^2)}=\mathcal{R}(UU^*)$ and $\big[|T^{t-1}|^2, UU^*\big]=0$.

Letting $t=2,3,\cdots, n+1$, respectively, we conclude that
$$A_1=B_1\Longleftrightarrow A_2=B_2\Longleftrightarrow\cdots\Longleftrightarrow A_{n+1}=B_{n+1}.$$
In view of \eqref{equ:the meaning of At equals Bt}, the proof of the equivalence of (i) and (ii) is complete.
\end{proof}

In view of Lemma~\ref{lem:Ito's 1st technical lemma}, we introduce the terms  of restricted sequence and the commutativity of an operator along a restricted sequence as follows:
\begin{definition}\label{defn: restricted sequence} A sequence $\{t_n\}_{n=1}^\infty$ is called restricted if
$t_n\in \{1,2,\cdots,n\}$ for each $n\in\mathbb{N}$, and an operator $T\in\mathcal{L}(H)$ is called commutative along this restricted sequence
if $T$ has the polar decomposition $T=U|T|$ such that
$$\Big[ U^{t_n}|T|\big(U^{t_n}\big)^*, \big|T^{n+1-t_n}\big|\Big]=0\quad \mbox{for any $n\in \mathbb{N}$.}$$
\end{definition}

A direct application of Lemma~\ref{lem:Ito's 1st technical lemma} and Definition~\ref{defn: restricted sequence} gives the following corollary:

\begin{corollary}\label{cor:an explanation of Ito's 1st technical lemma}\ Let $T\in\mathcal{L}(H)$ have the polar decomposition $T=U|T|$. Then the following statements are equivalent:
\begin{enumerate}
\item[{\rm (i)}] $\big[U^s|T|(U^s)^*, |T^t|\big]=0$ for any $s,t\in\mathbb{N}$;
\item[{\rm (ii)}] $T$ is commutative along any restricted sequence;
\item[{\rm (iii)}] $T$ is commutative along some restricted sequence.
\end{enumerate}
\end{corollary}

\begin{lemma}\label{lem:Ito's 2nd technical lemma}{\rm \cite[Lemma~4.3]{Ito}}\  Suppose that $T=U|T|$ is the polar decomposition of $T\in\mathcal{L}(H)$. Let $n\in\mathbb{N}$ be given such that \eqref{equ:commutative condition for first n terms} is satisfied. Then for $1\le k\le n+1$,
\begin{equation}\label{equ:form of square root of T k star-11}|(T^k)^*|=U|T|U^*\cdot U^2|T|(U^2)^*\cdot\ \cdots\ \cdot \ U^k|T|(U^k)^*.
\end{equation}
\end{lemma}
\begin{proof} This lemma was given in \cite[Lemma~4.3]{Ito}, where $H$ is a Hilbert space and $T\in\mathbb{B}(H)$.
 Checking the proof of \cite[Lemma~4.3]{Ito} carefully, we find out that the same is true for an adjointable operator on a Hilbert $C^*$-module.
\end{proof}

The main result of this section is as follows:
\begin{theorem}\label{thm:many equivalent conditions for centered operator}{\rm (cf.\,\cite[Theorem~4.1]{Ito})}\ Let $T\in\mathcal{L}(H)$ have the polar decomposition $T=U|T|$. Then the following statements are equivalent:
\begin{enumerate}
\item[{\rm (i)}] $T$ is a centered operator;
\item[{\rm (ii)}]$\big[|T^n|, |(T^m)^*|\big]=0$  for any $m,n\in\mathbb{N}$;
\item[{\rm (iii)}]$\big[|T^n|, |T^*|\big]=0$  for any $n\in\mathbb{N}$;
\item[{\rm (iv)}] $T$ is commutative along any restricted sequence;
\item[{\rm (v)}] $T$ is commutative along some restricted sequence;
\item[{\rm (vi)}]  $\big[U^m|T|(U^m)^*, |T^n|\big]=0$ for any $m,n\in\mathbb{N}$;
\item[{\rm (vii)}]$\big[U^n|T|(U^n)^*,|T|\big]=0$ for any $n\in\mathbb{N}$;
\item[{\rm (viii)}]$\big[|(T^n)^*|, |T|\big]=0$  for any $n\in\mathbb{N}$;
\item[{\rm (ix)}] $\big[(U^n)^*|T^*|U^n,|T^*|\big]=0$ for any $n\in\mathbb{N}$;
\item[{\rm (x)}]  $\big[(U^m)^*|T^*|U^m, |(T^n)^*|\big]=0$ for any $m,n\in\mathbb{N}$;
\item[{\rm (xi)}] The operators in
$\left\{|T|, U|T|U^*, U^*|T|U, U^2|T|(U^2)^*, (U^2)^*|T|U^2, \cdots\right\}$ commute with one another.
\end{enumerate}
\end{theorem}
\begin{proof}The proof of (i)$\Longleftrightarrow$(ii) is the same as that given in \cite[Theorem~4.1]{Ito}.
``(ii)$\Longrightarrow$(iii)" is clear by putting $m=1$ in (ii).

``(iii)$\Longleftrightarrow$(vii)":\  Putting $t_n=1$ and $s_n=n$ for any $n\in\mathbb{N}$. Then $T$ is commutative along $\{t_n\}\Longleftrightarrow$ (iii) is satisfied, and $T$ is commutative along $\{s_n\}\Longleftrightarrow$ (vii) is satisfied. The equivalence of (iii)--(vii) then follows from
Corollary~\ref{cor:an explanation of Ito's 1st technical lemma}.

``(vi)$\Longrightarrow$(ii)":\ Let $m$ and $n$ be any in $\mathbb{N}$. From (vii) and Lemma~\ref{lem:Ito's 2nd technical lemma} we know that $|(T^m)^*|$ has the form \eqref{equ:form of square root of T k star-11} with $k$ therein be replaced by $m$. Now, each term in \eqref{equ:form of square root of T k star-11} commutes with $|T^n|$ by (vi), so $\big[|T^n|, |(T^m)^*|\big]=0$.

The proof of the equivalence of (i)--(vii) is therefore complete. Since  $T$ is centered if and only if $T^*$ is centered, the equivalent conditions (viii), (ix) and (x) are then obtained by replacing $T$ and $U$ with $T^*$ and $U^*$, respectively.

It is obvious that (xi)$\Longrightarrow$(vii).

``(vii)+(ix)$\Longrightarrow$(xi)":\ From (vii), (ix) and Proposition~\ref{prop:commutative property extended-33}, we get
\begin{equation*}\label{equ:key commutative condition+2}\big[U^k|T|(U^k)^*,U^*U\big]=\big[(U^k)^*|T^*|U^k,UU^*\big]=0\quad \mbox{for any $k\in\mathbb{N}$}.
\end{equation*}
We prove that
the operators in
\begin{equation}\label{equ:defn of Omega}\Omega=\left\{|T|, U|T|U^*, U^*|T|U, U^2|T|(U^2)^*, (U^2)^*|T|U^2, \cdots\right\}\end{equation} commute with one another. That is,  $[A,B]=0$ for any $A,B\in \Omega$. To this end, four cases are considered as follows:

\textbf{Case 1:}\  $A=U^t|T|(U^t)^*$ and $B=U^s|T|(U^s)^*$ with $1\le t<s$. In this case, we have
\begin{eqnarray*}AB&=&U^t\cdot |T|\cdot U^{s-t}|T|(U^{s-t})^*\cdot (U^t)^*\\
&=&U^t\cdot U^{s-t}|T|(U^{s-t})^* \cdot  |T|\cdot (U^t)^*=BA.
\end{eqnarray*}

\textbf{Case 2:}\  $A=(U^t)^*|T|U^t$ and $B=(U^s)^*|T|U^s$ with $1\le t<s$. In this case, we have
$[A,B]=0$ as shown in Case 1 by replacing $U, T$ with $U^*, T^*$, since $A, B$ can be expressed alternately as $A=(U^{t+1})^*|T^*|U^{t+1}, B=(U^{s+1})^*|T^*|U^{s+1}$.

\textbf{Case 3:}\  $A=U^t|T|(U^t)^*$ and $B=(U^s)^*|T|U^s$ with $t,s\in\mathbb{N}$. In this case, we have
\begin{eqnarray*}BA&=&(U^s)^*|T|U^{s+t}|T|(U^t)^*=(U^s)^*\cdot |T|\cdot U^{s+t}|T|(U^{s+t})^*\cdot U^s\\
&=&(U^s)^*\cdot U^{s+t}|T|(U^{s+t})^* \cdot |T| \cdot U^s=U^t |T|(U^{s+t})^* \cdot |T| \cdot U^s=AB.
\end{eqnarray*}

\textbf{Case 4:}\ $A=|T|$ and $B=(U^s)^*|T|U^s$ with $s\in\mathbb{N}$. In this case, we have
\begin{eqnarray*}AB&=&U^*\cdot U|T|U^* \cdot (U^{s-1})^*|T|U^{s-1}\cdot U\\
&=&U^*\cdot(U^{s-1})^*|T|U^{s-1}\cdot U|T|U^*\cdot U\\
&=&(U^s)^*|T|U^s\cdot|T|U^*U=BA.
\end{eqnarray*}
This completes the proof that any two elements in
$\Omega$ are commutative.
\end{proof}

\textbf{Acknowledgments.} The authors thank the referee for very helpful comments and suggestions.

\bibliographystyle{amsplain}

\end{document}